\documentclass[10pt,preprint]{amsart}
\usepackage{amssymb,amsfonts,amsthm,amsmath,amscd,relsize}
\usepackage{hyperref}
\usepackage{enumerate}
\usepackage{comment}

 \usepackage{epsfig}
 \usepackage{graphicx}
 \usepackage{lipsum}
\usepackage{xcolor}
\usepackage{float}
\usepackage{soul}
\usepackage{adjustbox}

\usepackage{enumitem}

\usepackage{pgf,tikz,pgfplots}
\usepackage{xcolor}
\usetikzlibrary{calc}
\pgfplotsset{compat=1.13}
\usepackage{mathrsfs}
\usetikzlibrary{arrows}

\def\N{\mathbb N}
\def\Z{\mathbb Z}

\def\centerarc[#1](#2)(#3:#4:#5) { \draw[#1] ($(#2)+({#5*cos(#3)},{#5*sin(#3)})$) arc (#3:#4:#5); }

\theoremstyle{plain}
\newtheorem{theorem}{Theorem}

\newtheorem{corollary}{Corollary}

\theoremstyle{definition}

\theoremstyle{remark}

\begin{document}

\title{The sad life of lattice triangles}

\author{Christian Aebi and Grant Cairns}

\address{Coll\`ege Calvin, Geneva, Switzerland 1211}
\email{christian.aebi@edu.ge.ch}
\address{Department of Mathematical and Physical Sciences, La Trobe University, Melbourne, Australia 3086}
\email{G.Cairns@latrobe.edu.au}

\begin{abstract}
This paper treats triangles in the plane whose vertices lie on the integer lattice, i.e., the vertices have integer coordinates.  
It shows that apart from trivial examples, the circumcenter, centroid and orthocenter of such triangles never all lie on the integer lattice.
Several further observations are made concerning the circumcenter, centroid and orthocenter.
 \end{abstract}


\maketitle

A \emph{lattice triangle} $T$ is  a triangle in the plane whose vertices lie on the integer lattice, in other words, they have integer coordinates. Arguably, the most important points associated to any triangle are  its  \emph{circumcenter} $F$, i.e., the intersection of the perpendicular bisectors of the sides,  its \emph{centroid} $G$, i.e., the intersection of the medians, and its \emph{orthocenter} $H$, i.e., the intersection of the altitudes. The reader will recall that the three points $F,G,H$ lie on the famous \emph{Euler line} \cite[p.~71]{AN}. But it is a sad reality of the geometry of the integer lattice that these important points often fail to be lattice points themselves. In fact, the main point of this note is to observe that it \emph{never} happens that the three points $F,G,H$ are all lattice points, except in contrived circumstances.

 The centroid $G$ of our lattice triangle $T$ with vertices $v_1,v_2,v_3$ is $\frac13(v_1+v_2+v_3)$, where we are identifying a point with its position vector. So it is easy to see for a given triangle whether the centroid is a lattice point; this only happens  when  $v_1+v_2+v_3\equiv (0,0) \pmod 3$.
It is also well known that 
if the circumcenter $F$ of $T$ is a lattice point, then the orthocenter $H$ is also a lattice point. Indeed, since $3G=v_1+v_2+v_3$ is a lattice point and 
$H=3G-2F$ (see \cite[p.~71]{AN}),  $H$ is a lattice point when $F$ is a lattice point.
Nevertheless, it can happen that $H$ is a lattice point even when $F$ is not a lattice point. Figure~\ref{F:euler} shows an example of a right triangle for which $H$ is a lattice point but $F$ and  $G$ are not.

\begin{figure}[h]
\begin{center}
\begin{tikzpicture}[x=1.4cm,y=1.4cm]

\draw [fill, blue!20] (0,0) -- (2,0)-- (2,3) -- cycle ;
\draw [blue,semithick] (0,0) -- (2,0)-- (2,3) -- cycle ;

\draw[fill, black] (0,0) circle (.04);
\draw[fill, black] (2,0) circle (.04);
\draw[fill, black] (2,3) circle (.04);
\draw[fill, black] (1,3/2) circle (.04);
\draw[fill, black] (4/3,1) circle (.04);

\draw [black,font=\small]  (-.2,-.2) node {$O$} ;
\draw [black,font=\small]  (2.2,-.3) node {$v_1=H$} ;
\draw [black,font=\small]  (2.2,3) node {$v_2$} ;
\draw [black,font=\small]  (1.2,.8) node {$G$} ;
\draw [black,font=\small]  (.8,3/2) node {$F$} ;

\draw[-,blue,thick] (8/3,-1) -- (-2/3,4);

\draw[black,semithick] (1,3/2) circle (1.803);

\foreach \ii in {-1,...,3}
\draw[-,dotted,semithick] (\ii,-1) -- (\ii,4);
\foreach \jj in {-1,...,4}
\draw[-,dotted,semithick] (-1,\jj) -- (3,\jj);

\draw[->,semithick] (-1,0) -- (3,0);
\draw[->,semithick] (0,-1) -- (0,4);

  \end{tikzpicture}
\caption{A lattice triangle with Euler line $FGH$, where $H$ is a lattice
point but $F$ and $G$ are not.}\label{F:euler}
\end{center}
\end{figure}
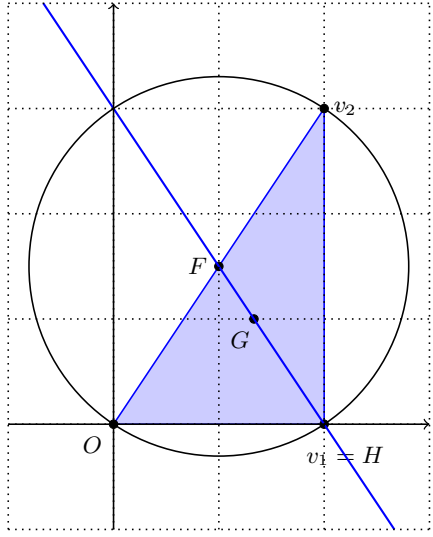
 
Before proceeding, we will need the formula for the circumcenter, which is conveniently expressed in complex numbers. Hence, we identify points in $\mathbb R^2$ with complex numbers. Nevertheless, in order to employ the vector cross product, we also interpret vectors in $\mathbb R^2$ as vectors in $\mathbb R^3$ with a $z$ coordinate of $0$.  It is convenient to be able to freely move among these interpretations of vectors. For example, for vectors $v, w$, we might consider both the cross product $v\times w$, and the complex multiplication $vw$.

By translation, we may suppose that $v_3$ is the origin $O$. 
We assume that our lattice triangle $T$ is non-degenerate, so the cross product $v_1\times v_2$ is nonzero. 
Let us write $v_1,v_2$ as complex numbers: $v_1 = x_1 + iy_1$, and $v_2 = x_2 + iy_2$, where $x_1,y_1,x_2,y_2 \in\Z$.
Observe that $v_1\overline{v_2}- \overline{v_1} v_2=2i(x_2 y_1 - x_1 y_2) $ and hence, taking norms,
\[
|v_1\overline{v_2}- \overline{v_1} v_2|=2|x_1 y_2-x_2 y_1 | = 2|v_1\times v_2|.
\]
Thus $v_1\overline{v_2}- \overline{v_1} v_2$ is nonzero and in particular, we may divide by $v_1\overline{v_2}- \overline{v_1} v_2$. 
We claim that the following formula gives the circumcenter of $T$.
\begin{equation}\label{E:Omega}
F:=\frac{v_1v_2(\overline{v_2}-\overline{v_1})}{v_1\overline{v_2}-\overline{v_1} v_2}.
\end{equation}
Indeed, the circumcenter is the center of a circle that
circumscribes the triangle, and thus, the distances between the vertices and
the circumcenter are radii, and hence all equal. Moreover, the circumcenter is the unique point with this property. So it suffices to show the point $F$ defined in \eqref{E:Omega} is equidistant from the vertices $O,v_1,v_2$ of $T$. 
The vector from $F$ to $v_1$ is
\[
\frac{v_1v_2(\overline{v_2}-\overline{v_1})}{v_1\overline{v_2}-\overline{v_1} v_2}-v_1=\frac{v_1\overline{v_2}(v_2-  v_1)}{v_1\overline{v_2}-\overline{v_1} v_2},
\]
which has the same length as $F$. Similarly, the vector from $F$ to $v_2$ also has the same length as $F$. So $F$ is the circumcenter of $T$, as claimed. 

As we saw above, the denominator in the expression for $F$ in \eqref{E:Omega} is
\begin{equation}\label{E:Omegadenom}
v_1\overline{v_2}-\overline{v_1} v_2=2i(x_2 y_1 - x_1 y_2).
\end{equation}
Expressed in Cartesian coordinates, the numerator for $F$ in \eqref{E:Omega} is
\begin{equation}\label{E:Omeganum}
v_1v_2(\overline{v_2}-\overline{v_1})=(x_2^2+y_2^2)(x_1,y_1) - (x_1^2+y_1^2)(x_2,y_2).
\end{equation}

From \eqref{E:Omegadenom} and \eqref{E:Omeganum}, the coordinates of $F$ are rationals. Since $G=\frac13(v_1+v_2)$, the coordinates of $G$ are also rationals. So we can simply multiply $v_1,v_2$ by an appropriate positive integer to obtain a lattice triangle for which the circumcenter $F$ and centroid $G$ are lattice points. Notice that since $H=3G-2F$, the orthocenter $H$ is then also a lattice point. For example, in Figure~\ref{F:euler}, the circumcenter $F$ is $(1,\frac32)$ and the centroid $G$ is $(\frac43,1)$. Dilating by a factor of $6$ gives a triangle whose circumcenter, centroid and orthocenter are lattice points; see Figure~\ref{F:euler2}. But this is cheating! The natural question is: is every such triangle constructed in this manner? In other words, is there a lattice triangle, with vertices $O,(x_1,y_1),(x_2,y_2)$, for which the circumcenter and centroid are lattice points, and $\gcd(x_1,y_1,x_2,y_2)=1$? It turns out, there is not.

\begin{figure}[h]
\begin{center}
\begin{tikzpicture}[x=1.8cm,y=1.8cm]

\draw [fill, blue!20] (0,0) -- (2,0)-- (2,3) -- cycle ;
\draw [blue,semithick] (0,0) -- (2,0)-- (2,3) -- cycle ;

\draw[fill, black] (0,0) circle (.04);
\draw[fill, black] (2,0) circle (.04);
\draw[fill, black] (2,3) circle (.04);
\draw[fill, black] (1,3/2) circle (.04);
\draw[fill, black] (4/3,1) circle (.04);

\draw [black,font=\small]  (-.2,-.2) node {$O$} ;
\draw [black,font=\small]  (2.7,-.2) node {$v_1=H=(12,0)$} ;
\draw [black,font=\small]  (2.8,3) node {$v_2=(12,18)$} ;
\draw [black,font=\small]  (2,1.) node {$G=(8,6)$} ;
\draw [black,font=\small]  (1.65,1.5) node {$F=(6,9)$} ;

\draw[-,blue,thick] (7/3,-.5) -- (-1/3,7/2);


\draw[black,semithick] (1,3/2) circle (1.803);

\foreach \ii in {-6,...,18}
\draw[-,dotted,semithick] (\ii/6,-.5) -- (\ii/6,3.5);
\foreach \jj in {-3,...,21}
\draw[-,dotted,semithick] (-1,\jj/6) -- (3,\jj/6);

\draw[->,semithick] (-1,0) -- (3,0);
\draw[->,semithick] (0,-.5) -- (0,3.5);

  \end{tikzpicture}
\caption{A lattice triangle for which $F,G, H$ are lattice points.}\label{F:euler2}
\end{center}
\end{figure}
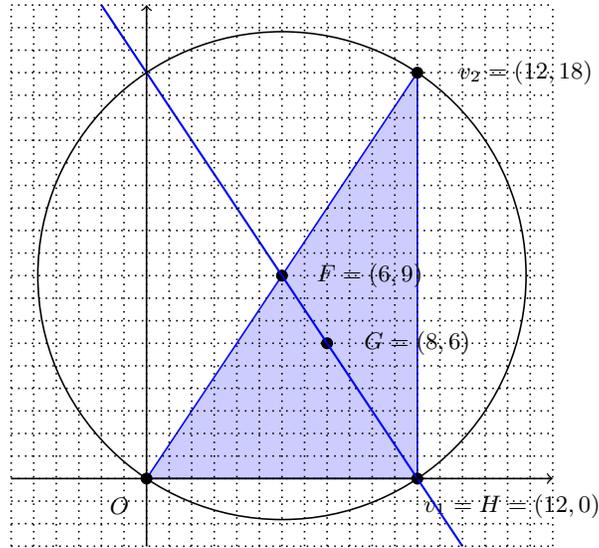

\begin{theorem}
If for the lattice triangle with vertices $O,(x_1,y_1),(x_2,y_2)$, the circumcenter  and the centroid are both lattice points, then $\gcd(x_1,y_1,x_2,y_2)$ is a multiple of $3$.
\end{theorem}

\begin{proof}
Naturally enough, the idea is  to work simply modulo $3$. If the centroid $G$ is lattice point, then 
\begin{equation}\label{E:xxyy}
x_1\equiv -x_2\pmod3\qquad\text{and}\qquad y_1\equiv -y_2\pmod3.
\end{equation}
Now looking at the formula for the circumcenter $F$,  and using the above relations,  we see the denominator of $F$ is $2(x_1 y_2-y_1 x_2) \equiv 0\pmod3$. As $F$ is a lattice point, the numerator of $F$ must then also be divisible by $3$.
From \eqref{E:Omeganum} and \eqref{E:xxyy}, the numerator, when evaluated modulo $3$,
 gives
\[
(x_2^2+y_2^2)(x_1,y_1) - (x_1^2+y_1^2)(x_2,y_2)\equiv 2(x_1^2+y_1^2)(x_1,y_1)\pmod3.
\]
Since squares are congruent to either $0$ or $1$ modulo $3$ according to whether or not the number is a multiple of $3$, it follows that if the sum of two squares is divisible by $3$ then both terms are multiples of $3$. So the above congruence implies necessarily that $x_1$ and $y_1$ are both divisible by $3$, and by  \eqref{E:xxyy} the same is true for $x_2$ and $y_2$, i.e., $\gcd(x_1,y_1,x_2,y_2)$ is a multiple of $3$.
\end{proof}

If the circumcenter is a lattice point but the centroid isn't, the conclusion to Theorem 1 may not hold. Nevertheless, in this case, a different property is enjoyed by the coordinates  $x_1,y_1,x_2,y_2$. 

\begin{theorem}
Consider a lattice triangle $T$ with vertices $O,v_1=(x_1,y_1),v_2=(x_2,y_2)$. 
If the circumcenter $F$ of  $T$ is a lattice point, then $x_1+y_1$ and  $x_2+y_2$ are both even.
\end{theorem}

\begin{proof} Let $\sigma:\Z^2\to\Z_2$ be the function defined by setting $\sigma(x,y)$ to be the reduction of $x+y$ modulo $2$,  for all $(x,y)\in\Z^2$. 
If $\sigma(v_1) = \sigma(v_2) = 0$, we are done. This leaves three potential cases: $(\sigma(v_1),\sigma(v_2)) = (0,1), (1,0)$ or $(1,1)$. Note that we may assume $\sigma(v_1) = 0$ and $\sigma(v_2) = 1$. 
Indeed, if $\sigma(v_1) = 0$ and $\sigma(v_2) = 1$, then leave the vertices unchanged.
If $\sigma(v_1) = 1$ and $\sigma(v_2) = 0$, then interchange the labels $v_1,v_2$, so that one obtains $\sigma(v_1) = 0$ and $\sigma(v_2) = 1$.
If $\sigma(v_1) = \sigma(v_2) = 1$, then $\sigma(v_2-v_1)=0$, since $\sigma$ is a group homomorphism. In this case, translate $T$ so that $v_2$ is moved to the origin and relabel by $v_2$ the vertex that was previously at the origin. Then $\sigma(v_1)=0$ and $\sigma(v_2)=1$. Hence, it suffices to treat the case where $\sigma(v_1)=0$ and $\sigma(v_2)=1$. We will show that this leads to a contradiction.

From \eqref{E:Omega} and \eqref{E:Omegadenom}, 
\begin{equation}\label{E:Omega2}
2i(x_2 y_1 - x_1 y_2)F=v_1v_2(\overline{v_2}-\overline{v_1}).
\end{equation}
Let $k\in\N\cup\{0\}$ be the largest power of $2$ that divides $v_1$, say $v_1=2^kv'_1$ where $v'_1=x'_1+iy'_1$ with $x'_1,y'_1\in\Z$, so $x_1=2^kx'_1,y_1=2^ky'_1$, and $x'_1,y'_1$ are not both even. From \eqref{E:Omega2}, 
\begin{equation}\label{E:Omega3}
2i(x_2 y'_1 - x'_1 y_2)F=v'_1v_2(\overline{v_2}-\overline{v_1}).
\end{equation}

This is an equation in the Gaussian integers $\Z[i]$, \index{integers!Gaussian} 
since $F\in \Z^2$ by hypothesis. We will employ the fact that the Gaussian integers enjoy unique prime factorization; see \cite[Chapter~7]{Stillwell}.
Note that in the Gaussian integers, the prime factorization of $2$ is $2=(1+i)(1-i)$. So by  \eqref{E:Omega3}, $1+i$ is a factor of either $v'_1,v_2$ or $\overline{v_2}-\overline{v_1}$, and in the latter case, $1-i$ is a factor of  $v_2- v_1$.
But observe that if a Gaussian integer $v$ is divisible by $1+i$, say $v=(a+ib)(1+i)$, for some $a,b\in\Z$, then $v=(a-b)+i(a+b)$, so $\sigma(v)=0$. 
If $v$ is divisible by $1-i$, say $v=w(1-i)$, then $v=(-iw)(1+i)=(a+ib)(1+i)$, for some $a,b\in\Z$, and once again, $\sigma(v)=0$. Thus, since $\sigma(v_2)=\sigma(v_2-v_1)=1$ by assumption, $1+i$ must be a factor of $v'_1$. By the same reasoning, $1-i$ is also a factor of $v'_1$. But then $v'_1$ is divisible by $2$, contradicting the fact that $x'_1,y'_1$ are not both even.
\end{proof}

The area of a lattice triangle $T$ is an integer multiple of 1/2. The above theorem has the following consequence.

\begin{corollary} 
If the circumcenter $F$ of a lattice triangle $T$ is a lattice point, then the area of $T$ is an integer. 
\end{corollary}

\begin{proof}
From the above theorem, if $T$ has vertices $O,v_1=(x_1,y_1),v_2=(x_2,y_2)$, then $x_1\equiv-y_1 \pmod 2 $ and  $x_2\equiv-y_2\pmod 2 $. The area of $T$ is 
$\frac12|v_1\times v_2|=\frac12|x_1 y_2-x_2 y_1|$. So, since $x_1 y_2-x_2 y_1\equiv 0\pmod 2 $, the corollary follows.
\end{proof}

We now give a few examples that highlight the reluctance of the three canonical centers (the centroid, the circumcenter and the orthocenter) to be lattice points. First note however that for each $n\in\N$, there is a lattice triangle of area $n$ whose circumcenter is a lattice point.
Indeed, consider the lattice triangle $T$ with vertices $O,(2,0),(n,n)$. It has area $n$ and circumcenter $(1,n-1)$.

For the lattice triangle $T$ with vertices $O,(4,2),(1,5)$,  its Euler line contains no lattice points whatsoever. Indeed, its  centroid is $\frac13(5,7)$ and its circumcircle is $\frac13(4,7)$, 
so the Euler line  is the horizontal line through $(0,\frac73)$; see Figure~\ref{F:badeuler}. 
It can happen that the circumcenter, centroid and orthocenter all fail to be lattice points even when the circumradius is an integer, the area of the triangle is an integer and the coordinates of each side vector  have even sum. Indeed,
this is the case for the lattice triangle with vertices $O,(19,17),(11,23)$, which has circumcenter $\frac15(39,52)$, centroid $(10,\frac{40}3)$, orthocenter $\frac15(72, 96)$. It has area $125$, circumradius $13$ and the coordinates of each side vector have even sum.

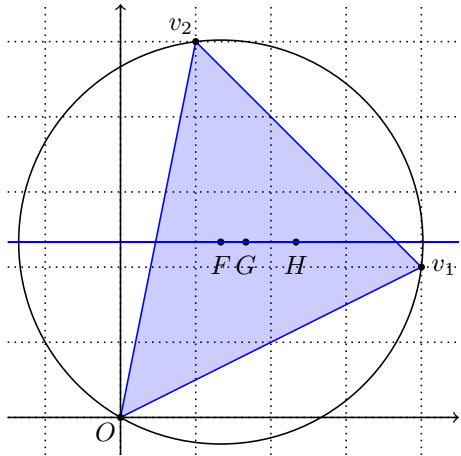
\begin{figure}
\begin{center}
\begin{tikzpicture}

\draw [fill, blue!20] (0,0) -- (4,2)-- (1,5) -- cycle ;
\draw [blue,semithick] (0,0) -- (4,2)-- (1,5) -- cycle ;

\draw[fill, black] (0,0) circle (.04);
\draw[fill, black] (4,2) circle (.04);
\draw[fill, black] (1,5) circle (.04);
\draw[fill, black] (5/3,7/3) circle (.04);
\draw[fill, black] (4/3,7/3) circle (.04);
\draw[fill, black] (7/3,7/3) circle (.04);

\draw [black]  (-.2,-.2) node {$O$} ;
\draw [black]  (4.3,2) node {$v_1$} ;
\draw [black]  (7/3,7/3-.3) node {$H$} ;
\draw [black]  (.8,5.2) node {$v_2$} ;
\draw [black]  (5/3,7/3-.3) node {$G$} ;
\draw [black]  (4/3,7/3-.3) node {$F$} ;

\draw[-,blue,thick] (-1.5,7/3) -- (4.5,7/3);

\draw[black,semithick] (4/3,7/3) circle (2.687);

\foreach \ii in {-1,...,4}
\draw[-,dotted,semithick] (\ii,-.5) -- (\ii,5.5);
\foreach \jj in {0,...,5}
\draw[-,dotted,semithick] (-1.5,\jj) -- (4.5,\jj);

\draw[->,semithick] (-1.5,0) -- (4.5,0);
\draw[->,semithick] (0,-.5) -- (0,5.5);

  \end{tikzpicture}
\caption{A lattice triangle whose Euler line contains no lattice points.}\label{F:badeuler}
\end{center}
\end{figure}

An example of a lattice triangle whose circumcenter and orthocenter are lattice points, but whose centroid is not a lattice point,
is given by the lattice triangle with vertices $O,(6,0),(8,4)$. Its centroid is $\frac{1}3(14,4)$, while its circumcenter is $(3,4)$ and its orthocenter is $(8,-4)$; see Figure~\ref{F:badcent}. Notice that in this example, the circumradius is an integer (it is $5$), the area of the triangle is also an integer (it is $12$), and the coordinates of each side vector have even sum.

\begin{comment}
An example of a lattice triangle whose circumcenter and orthocenter are lattice points, but whose centroid is not a lattice point,
is given by the lattice triangle with vertices $O,(10,0),(2,4)$. Its centroid is $(4,\frac43)$, while its circumcenter is $(5,0)$ and its orthocenter is the vertex $v_2=(2,4)$; see Figure~\ref{F:badcent}.

\begin{figure}
\begin{center}
\begin{tikzpicture}[x=.5cm,y=.5cm]
\draw [fill, blue!20] (0.,0.) -- (10.,0.) -- (2.,4.) -- cycle;
\draw [blue] (0.,0.) -- (10.,0.) -- (2.,4.) -- cycle;

\draw [black] (5.,0.) circle (5.cm);

\draw [-,blue,thick,domain=1.2:8.8] plot(\x,{(-20.--4.*\x)/-3.});

\draw [fill=blue] (0.,0.) circle (2pt);
\draw[color=black] (-0.4,-0.4) node {$O$};
\draw [fill=black] (10.,0.) circle (2pt);
\draw[color=black] (10.5,0.4) node {$v_1$};
\draw [fill=black] (2.,4.) circle (2pt);
\draw[color=black] (1.37,4.09) node {$v_2$};
\draw [fill=blue] (4.,1.3) circle (2pt);
\draw[color=black] (4.5,1.6) node {$G$};
\draw [fill=black] (5.,0.) circle (2pt);
\draw[color=black] (5.4,0.5) node {$F$};
\draw [fill=blue] (2.,4.) circle (2pt);
\draw[color=black] (2.9,4.05) node {$H$};
 
\foreach \ii in {-1,...,11}
\draw[-,dotted] (\ii,-5) -- (\ii,5);
\foreach \jj in {-5,...,5}
\draw[-,dotted] (-1,\jj) -- (11,\jj);

\draw[->] (-1,0) -- (11,0);
\draw[->] (0,-5) -- (0,5);

 \end{tikzpicture}
\caption{A lattice triangle whose circumcenter and orthocenter are lattice points, but whose centroid is not.}\label{F:badcent}
\end{center}
\end{figure}
\end{comment}

\begin{figure}
\begin{center}
\begin{tikzpicture}[x=.5cm,y=.5cm]
\draw [fill, blue!20] (0.,0.) -- (6.,0.) -- (8.,4.)  -- cycle;
\draw [blue,semithick] (0.,0.) -- (6.,0.) -- (8.,4.)  -- cycle;

\draw [black,semithick] (3.,4.) circle (2.5cm);

\draw [-,blue,thick,domain=-.75:8.6] plot(\x,{(44.-8.*\x)/5.});

\draw [fill=black] (0.,0.) circle (2pt);
\draw[color=black] (-0.5,-0.5) node {$O$};
\draw [fill=black] (6.,0.) circle (2pt);
\draw[color=black] (6.5,-0.5) node {$v_1$};
\draw [fill=black] (8.,4.) circle (2pt);
\draw[color=black] (8.5,3.5) node {$v_2$};
\draw [fill=black] (4.667,1.33) circle (2pt);
\draw[color=black] (5.118172505383182,1.6) node {$G$};
\draw [fill=black] (3.,4.) circle (2pt);
\draw[color=black] (3.4,4.4) node {$F$};
\draw [fill=black] (8.,-4.) circle (2pt);
\draw[color=black] (8.4,-3.6) node {$H$};
 
\foreach \ii in {-3,...,9}
\draw[-,dotted,semithick] (\ii,-5) -- (\ii,10);
\foreach \jj in {-5,...,10}
\draw[-,dotted,semithick] (-3,\jj) -- (9,\jj);

\draw[->,semithick] (-3,0) -- (9,0);
\draw[->,semithick] (0,-5) -- (0,10);

\end{tikzpicture}
\caption{A lattice triangle whose circumcenter $F$ and orthocenter $H$ are lattice points, but whose centroid $G$ is not.}\label{F:badcent}
\end{center}
\end{figure}
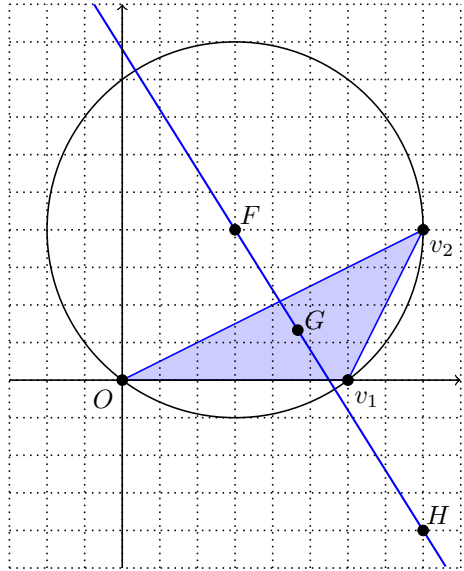

The converse to the above theorem is false. Indeed, consider the lattice triangle $T$ with vertices $O,(2,0),(1,3)$. Its centroid is the lattice point $(1,1)$. The coordinates of each side vector have even sum (and the area of $T$ is an integer),
however, the circumcenter of $T$ is $(1,\frac43)$, which is not a lattice point. 

An example of a lattice triangle whose centroid and orthocenter are lattice points, but whose circumcenter is not a lattice point,
is given by the lattice triangle $T$ with vertices $O,(3,3),(3,18)$. The circumcenter of $T$ is $(-\frac{15}2,\frac{21}2)$, while its orthocenter is $(21,0)$.
The circumcenter of  lattice triangle may fail to be a lattice point, even when the centroid and orthocenter are lattice points, the coordinates of each side vector have even sum, and the area of $T$ is an integer. An example is the lattice triangle $T$ with vertices $O,(12,6),(12,18)$. 
Its circumcenter of $T$ is $(\frac32,12)$, while its orthocenter is $(21,0)$.

Notice also that for the triangle of Figure~\ref{F:euler2}, the circumcenter, centroid and orthocenter are all lattice point,  the area of the triangle is an integer,
 the coordinates of each side vector  have even sum, but the circumradius is not an integer; it is $3\sqrt{13}$. 

In the above discussion, we have considered six conditions: the circumcenter $F$ is a lattice point, the centroid $G$ is a lattice point, the orthocenter $H$ is a lattice point, the circumradius is an integer, the area of the triangle is an integer, and the coordinates of each side vector  have even sum.  We have not made a systematic study of which subsets of the set of these six conditions are implied by which other subsets. Perhaps this would be a nice student project. In any case, there is one more implication that we would like to highlight.

\begin{theorem}
If for a lattice triangle $T$,  the orthocenter $H$ is a lattice point, and the circumradius is an integer, then the circumcenter $F$ is a lattice point.  \end{theorem}

\begin{proof} 
First translate $T$ so that one of its vertices is moved to the origin. Because of the equation $ F=(3G-H)/2$,  if $H$ is a lattice point, then $F$ has the form $(a,b)/2$ where $a,b\in \Z$. Then,  since the origin is a vertex, the circumradius is $\frac12 \sqrt{a^2+b^2}$. But this can only be an integer if $a$ and $b$ are both even. Hence $F$ is a lattice point. 
\end{proof}


\bibliographystyle{vancouver}

\end{document}